\documentclass[reqno]{amsart}
\usepackage{amssymb, amsthm, amsmath}
\usepackage{comment}
\usepackage{graphicx}
\usepackage{enumitem}
\usepackage{hyperref}
\usepackage{xcolor}
\usepackage{changepage}

\newtheorem{theorem}{Theorem}
\newtheorem{lemma}[theorem]{Lemma}
\newtheorem{corollary}[theorem]{Corollary}
\newtheorem{conjecture}[theorem]{Conjecture}
\newtheorem{proposition}[theorem]{Proposition}

\theoremstyle{definition}

\newtheorem{remark}[theorem]{Remark}

\theoremstyle{remark}

\newcommand{\R}{\mathbb{R}}

\newcommand{\Z}{\mathbb{Z}}
\newcommand{\N}{\mathbb{N}}

\newcommand{\PP}{\mathbb{P}}
\newcommand{\EE}{\mathbb{E}}

\newfont{\cmbsy}{cmbsy10}
\newfont{\cmmib}{cmmib10}

\newcommand{\Mod}[1]{\ (\mathrm{mod}\ #1)}

\DeclareMathOperator{\supp}{supp}

\newcommand{\Addresses}{{
		\footnotesize
		\bigskip
		\footnotesize
	    \textsc{Chennai Mathematical Institute, H1, SIPCOT IT Park, Siruseri, Kelambakkam 603103, India}\par\nopagebreak	
		\textsc{Department of Mathematics and Statistics, Queen's University, Kingston, Ontario, K7L 3N6, Canada}\par\nopagebreak
		\textit{E-mail addresses:}
		\texttt{bharadwaj.work@outlook.com,} \texttt{brad.w.rodgers@gmail.com}
		
}}

\title[Large prime factors]{Large prime factors of well-distributed sequences}
\author{Abhishek Bharadwaj, Brad Rodgers}
\date{}

\begin{document}
	\thispagestyle{empty}
	
	\begin{abstract}
		We study the distribution of large prime factors of a random element $u$ of arithmetic sequences satisfying simple regularity and equidistribution properties. We show that if such an arithmetic sequence has level of distribution $1$ the large prime factors of $u$ tend to a Poisson-Dirichlet process, while if the sequence has any positive level of distribution the correlation functions of large prime factors tend to a Poisson-Dirichlet process against test functions of restricted support. For sequences with positive level of distribution, we also estimate the probability the largest prime factor of $u$ is greater than $u^{1-\epsilon}$, showing that this probability is $O(\epsilon)$.
		
		Examples of sequences described include shifted primes and values of single-variable irreducible polynomials. 
		
		The proofs involve (i) a characterization of the Poisson-Dirichlet process due to Arratia-Kochman-Miller and (ii) an upper bound sieve.
	\end{abstract}

        \maketitle
	
	\section{Introduction}
	\label{sec:intro}
	
	\subsection{Background}
	\label{subsec:background}
	The purpose of this note is to study the distribution of large prime factors of elements in sequences which satisfy only a few minimal conditions.
	
	Let us recall the classical theory and notation: let $P^+(u)$ be the largest prime factor of a positive integer $u$. It is a consequence of a result of Dickman \cite{KDi} that for any fixed $c \in (0,\infty)$,
	$$
	\frac{1}{x}|\{u\leq x:\, P^+(u) \leq u^c\}| \sim \rho(1/c),
	$$
	as $x\rightarrow\infty$, where $\rho:(0,\infty) \rightarrow (0,1]$ is a continuous function called the Dickman function. (See e.g. \cite[Ch. 7.1]{MoVa} for a modern account.)
	
	This result was generalized and given a probabilistic interpretation by Billingsley \cite{Bi72} (and independently Knuth and Trabb Pardo \cite{KnTr} and Vershik \cite{Ve}). Let $u$ be chosen randomly and uniformly from the integers from $1$ to $x$. Then the above result says that $\log P^+(u)/\log u$ tends in distribution to a nonnegative random variable $L_1$ with cumulative distribution function $\PP(L_1\leq c) = \rho(1/c)$. Moreover, let $p_1 \geq p_2 \geq \cdots $ be the prime factors of $u$ listed with multiplicity, with the convention that $p_j = 1$ if $u$ has fewer than $j$ prime factors (so that $p_1 = P^+(u)$ and $u = p_1p_2\cdots$). Billingsley showed that there is a sequence of (dependent) random variables $L_1, L_2,...$ such that for any $k \geq 1$, and any fixed constants $c_1,c_2,...,c_k \in [0,\infty)$,
	$$
	\mathbb{P}\Big(\frac{\log p_1}{\log u} \leq c_1,..., \frac{\log p_k}{\log u} \leq c_k\Big) \rightarrow \mathbb{P}(L_1 \leq c_1,...,L_k \leq c_k),
	$$
	as $x\rightarrow\infty$. That is, the process $\tfrac{\log p_1}{\log u}, \tfrac{\log p_2}{\log u},...$ tends in distribution to $L_1, L_2,...$. (See e.g. \cite{Bi99} for a modern probabilistic account. Here and in what follows we discard with the case $u=1$ by adopting the formal convention $\log 1 / \log 1 = 1$.)
	
	For each $k$ an explicit formula for $\PP(L_1 \leq c_1,...,L_k \leq c_k)$ can be written down (see \cite[Thm. 4.4]{Bi99} for a density formula), but the formula is somewhat complicated. The following characterization is simpler: let $U_1, U_2, ...$ 
	be independent and identically distributed uniform random variables on the interval $(0,1)$, and define random variables $G_1 = 1- U_1$, $G_2 = U_1(1-U_2)$, $G_3 = U_1 U_2 (1-U_3)$, ...\,. Then $L_1 \geq L_2 \geq ...$ may be defined as the outcome of sorting $G_1, G_2,...$ into nonincreasing order. 
	
	The sequence of random variables $L_1, L_2, ...$ is known as the \emph{Poisson-Dirichlet process with parameter $\theta = 1$}. As the name suggests there are Poisson-Dirichlet processes with $\theta \neq 1$, but in this paper we will only deal with $\theta = 1$, and so if there is no risk of confusion, {we} refer to $L_1, L_2, ...$ as simply the Poisson-Dirichlet process. (See \cite{Bi99,Ki} for an account of general $\theta$, along with a more detailed introduction to the case $\theta = 1$.)
	
	We note that $\tfrac{\log p_1}{\log u} + \tfrac{\log p_2}{\log u} + \cdots = 1$, and likewise $L_1 + L_2 + \cdots = 1$ almost surely.
	
	\subsection{Some well-distributed arithmetic sequences}
	\label{subsec:welldistributed}
	
	It is natural to wonder whether this statistical pattern governing the distribution of large prime factors of random integers extends to other arithmetic sequences. Some cases which have been studied extensively include shifted primes \cite{BaHa,FeWu,Fo,Go,Ho73} -- that is, the sequence $\{p-a\}$ for a constant $a$ where $p$ ranges over the primes -- and the values of irreducible polynomials {\cite{Da,DaMa,DaMaTe,dlBr,dlBDr,DeIw,He,Ho67,Ir,Me}} -- that is the sequence $\{F(n)\}$, where $F$ is an irreducible polynomial with integer coefficients and a positive leading coefficient and $n$ ranges over the integers.
	
	In this paper we study a quite general class of arithmetic sequences. Let $(a_n)$ be a sequence of nonnegative numbers. Define the quantities
	\begin{equation}
		\label{eq:N_def}
		N(x) = \sum_{n\leq x} a_n, \quad\quad N_d(x) = \sum_{\substack{n\leq x \\ n \equiv 0 \Mod d}} a_n.
	\end{equation}
	We say that the sequence $(a_n)$:
	\begin{enumerate}[label = {(\Alph*)}]
		\item \label{A} \label{PROP=index} {Has \emph{index} $\alpha$ if for some $\alpha > 0$,
        $$
        N(x) = x^{\alpha+o(1)},
        $$}
		\item Has a \emph{level of distribution $\vartheta$} if for any $0 < c < \vartheta$ and any $A > 0$,
        \begin{equation}\label{eq:distribution-eqn}
            \sum_{d\leq x^c} |N_d(x) - g(d) N(x)| \ll_{c,A} \frac{1}{(\log x)^A} N(x),
        \end{equation}		
		where $g(d)$ is a multiplicative function with $g(d) \in [0,1]$ for all $d\geq 1$, and
		\begin{equation}
			\label{eq:g_asymp}
			\sum_{p \leq x} g(p) \log p = \log x + O(1), \quad \quad g(d) = O(C^{\Omega(d)}/d),
		\end{equation}
		for some constant $C \geq 1$.
		\item {\label{PROP=Locally-Uniform}} Is {\emph{congruence uniform} if there are constants $B\geq 0$, $C \geq 1$ such that
		$$
		N_d(x) \ll \Big(\frac{C^{\Omega(d)}}{d} N(x) + C^{\Omega(d)} \Big) (\log x)^B,\quad \textrm{for}\; d\leq x.
		$$
        }
	\end{enumerate}

    {Additionally we say the sequence $(a_n)$:
    \begin{itemize}
        \item Is $\sigma$ \textrm{\emph{well-distributed}} if $(a_n)$ satisfies each of (A), (B), and (C) and $\sigma$ is any positive number less than or equal to $\alpha$ and $\vartheta$.
    \end{itemize}}

	Here as usual $\Omega(d)$ is the number of prime factors of $d$ counted with multiplicity. Note that \eqref{eq:g_asymp} implies that $g(p^k) < 1$ except for finitely many primes $p$.

	Likewise note that we trivially have $N_d(x) \leq N(x)$, so that the condition for congruence uniformity is only meaningful for those integers $d$ with prime factors $p > C$.
	
	It will typically be the case that $a_n$ is the indicator function of $n$ belonging to some subset of the integers, and in that case we will also describe that subset by the above terminology as long as there is no chance for confusion.
    
    \begin{remark}
    \label{remark: ABC}
    {Let us explain some intuition for these conditions. (A) is self-explanatory;} in (B) one may keep in mind the examples $g(d) = 1/d$ or $g(d) = 1/\phi(d)$; and (C) may be thought of as a more technical condition -- the reader may check that it is trivially satisfied if the sequence $(a_n)$ is bounded and $N(x) \gg x/(\log x)^B$ for some $B$.
    \end{remark}
	
	Examples of sequences satisfying these conditions are described by the following propositions. Throughout the paper, for a proposition $\mathfrak{A}$, we use the notation $\mathbf{1}[\mathfrak{A}]$ to be $1$ if $\mathfrak{A}$ is true and $0$ otherwise.
	
	\begin{proposition}
		\label{prop:shiftedprimes_dist}
		{Shifted primes are $1/2$ well-distributed.} That is: for a fixed integer $a$, consider the set $\mathcal{B} = \{p-a:\, p\; \textrm{is prime}\}$ and let $a_n = \mathbf{1}[n \in \mathcal{B}]$. Then $(a_n)$ has index $1$, has level of distribution $\vartheta = 1/2$, and is congruence uniform.
	\end{proposition}
	
	\begin{proof}
		{That the sequence has index $\alpha = 1$ follows} from the prime number theorem (or indeed Chebyshev's bounds). {By Remark \ref{remark: ABC}, congruence uniformity also follows from Chebyshev's lower bound $\pi(x) \gg x/\log x$.} {The} level of distribution $1/2$ follows from the Bombieri-Vinogradov theorem \cite[Thm. 9.2.1]{CoMu}, with {$g(d) = \mathbf{1}_{(a,d)=1}/\phi(d)$, where $\mathbf{1}_{(a,d)=1}$ is $1$ when $d$ is coprime to $a$ and $0$ otherwise.}
	\end{proof}
	
	\begin{proposition}
		\label{prop:polynomials_dist}
        {The values of irreducible polynomials of degree $D$ are $1/D$ well-distributed.} That is: for a polynomial $F(X) \in \Z[X]$ of degree $D\geq 1$ which is irreducible with positive leading coefficient, consider the set $\mathcal{C} = \{ F(n):\, n \in \N_{\geq 1} \} \cap \N_{\geq 1}$ and let $a_n = \mathbf{1}[n \in \mathcal{C}]$. Then $(a_n)$ {has index $1/D$}, has level of distribution $\vartheta = 1/D$, and is congruence uniform. 
	\end{proposition}
	\begin{proof}
	{That the index is $1/D$} follows from the fact that $N(x) \asymp x^{1/D}$. 
		
		We see that the level of distribution is $1/D$ in the following way. For a natural number $d$, let $h(d)$ be the number of distinct roots of $F$ modulo $d$. By the Chinese Remainder Theorem (see \cite[Theorem 46]{Nagell}) we note that $h$ is multiplicative. Set $g(d) = h(d)/d$ and note $g$ is multiplicative also. 
  
        Obviously $g(d) \in [0,1]$. We will show that \eqref{eq:distribution-eqn} and \eqref{eq:g_asymp} are satisfied for this function $g$. Let us show that \eqref{eq:g_asymp} holds first. It is known that $\sum_{p\leq x} g(p) \log p = \log x + O(1)$; this is just a weak claim that $F$ will on average have one root modulo a prime, see e.g. \cite[ Eq. (4), Pg. 352]{Na}, or \cite[Corollary 4]{Le} for a more modern statement. 
        
        The second claim in \eqref{eq:g_asymp} is just the claim that $h(d) \ll C^{\Omega(d)}$ for some constant $C > 1$. {This follows from the multiplicativity of $h$ and \cite[Theorem 54]{Nagell} by taking $C=D \cdot \mathrm{disc}(F)^2$}.

        Now, let us now prove \eqref{eq:distribution-eqn}. {We first establish that
        \begin{equation}
        \label{eq: pointwise_diff}
        |N_d(x) - g(d)N(x)| \ll h(d),
        \end{equation}
        where the implicit constant depends only on $F$.
        
        To see this, note that since $F(n+d) \equiv F(n) \bmod d$, in every interval of length $d$ we will have $F(n) \equiv 0 \Mod d$ for $h(d)$ values of $n$. In particular since $F$ is eventually increasing, there is some sufficiently large $x_0$ (depending on $F$) such that whenever $N(x)$ increases by $d$ on an interval to the right of $x_0$, $N_d(x)$ will increase by $h(d)$ on the same interval. Therefore 
        $$
        |N_d(x) - g(d)N(x)| \leq h(d) + |N_d(x_0) - g(d)N(x_0)| = h(d) + O(1)
        $$
        with the $O(1)$ term present to account for behavior of this quantity for $x \leq x_0$.

        If $h(d) \geq 1$, this verifies \eqref{eq: pointwise_diff}. On the other hand if $h(d) = 0$, we have $N_d(x) = g(d) N(x) = 0$ for all $x$, so \eqref{eq: pointwise_diff} is verified in this case also.}

        Hence from \eqref{eq: pointwise_diff},
        \begin{equation}
        \label{eq: average_diff}
        \sum_{d \le x^c} |N_d(x) - g(d)N(x)| \ll \sum_{d \le x^c}h(d) \ll_{c} x^c (\log x)^A,
        \end{equation}
        where we use Lemma \ref{lem:estimate-C}, proved below, in the last estimate. As $N(x) \asymp x^{1/D}$, we see \eqref{eq:distribution-eqn} is satisfied as long as $c < 1/D$.
        
        {To prove congruence uniformity, we use the bound {$ N_d(x) = g(d)N(x) + O(h(d))$}. Consequently for a constant $C > 1$ as above, we see that $N_d(x) \ll \frac{C^{\Omega(d)}}{d} N(x) + C^{\Omega(d)}$, for all $d \leq x$.}
	\end{proof}
        
        We have used one of the following results above and we will need them later as well. 

        \begin{lemma}
        \label{lem:estimate-C}
        Suppose $h(d)$ is a multiplicative function and $g(d) = h(d)/d$ satisfies $g(d) \in [0,1]$ for all $d$ and $g(d) = O(C^{\Omega(d)}/d)$ for some constant $C \geq 1$. Then for some constant $A > 0$, 
        $$
        \sum_{n \le x} g(n) \ll (\log x)^A
        $$ 
        and 
        $$
        \sum_{n \le x} h(n) \ll x (\log x)^A.
        $$
        \end{lemma}

        \begin{proof}
        Let $p_1,...,p_\ell$ be the finite set of primes less than $2C$, and set $P = p_1\cdots p_\ell$. Then
		\begin{multline*}
		\sum_{\substack{n\leq x \\ (n,P)=1}} g(n) \ll \sum_{\substack{n\leq x \\ (n,P)=1}} C^{\Omega(n)}/n \leq \prod_{\substack{p \leq x \\ (p,P)=1}} \Big(1 + \frac{C}{p} + \frac{C^2}{p^2}+\cdots\Big) \\
		= \exp\Big(\sum_{p\leq x} \frac{C}{p} + O(1)\Big) \ll (\log x)^C,
		\end{multline*} 
		where we have used the fact that $C/p$ in the product above is no more than $1/2$ in order to sum the series. Using multiplicativity and $g(p_i^{e_i}) \in [0,1]$ for $1\leq i \leq \ell$, we have from a crude bound
		\begin{multline*}
		\sum_{n\leq x} g(n) \le \sum_{p_1^{e_1}\cdots p_\ell^{e_\ell} \leq x} \;\sum_{\substack{m \leq x/p_1^{e_1}\cdots p_\ell^{e_\ell}\\ (m,P)=1}}g(m) \\
		\ll (\log x)^C \Big(\sum_{e_1:\; p_1^{e_1}\leq x} 1\Big) \cdots \Big(\sum_{e_\ell:\; p_\ell^{e_\ell}\leq x} 1\Big) \ll (\log x)^{C+\ell}.
		\end{multline*}
		This proves the first estimate. For the second, we have
		$$
		\sum_{n\leq x} h(n) \leq x \sum_{n\leq x} g(n),
		$$
		which implies the claim.
        \end{proof}
	
	Regarding the level of distribution of shifted primes, one expects more can be said:
	
	\begin{conjecture}
		\label{conj:shiftedprimes_dist}
		The shifted primes have level of distribution $\vartheta = 1$.
	\end{conjecture}
	
	Indeed this is a slightly weaker version of the Elliott-Halberstam conjecture \cite[Ch. 9.2]{CoMu}. 
	
	On the other hand it does not seem that the values of irreducible polynomials of degree $2$ or greater have level of distribution $1$. 
	
	Some interesting arithmetic sequences are known to have level of distribution $1$ however, for instance those positive integers $\mathcal{S}$ which indicate a \texttt{0} in the Thue-Morse sequence \cite{Sp}. $\mathcal{S}$ may be characterized in the following way: it is the collection of positive integers $n$, where $n$ has an even number of \texttt{1}s in its binary expansion.
	
	\begin{proposition}
		\label{prop:thuemorse_dist}
		{The values of the Thue-Morse sequence are $1$ well-distributed. That is, let $a_n = \mathbf{1}[n \in \mathcal{S}]$. Then $(a_n)$ has index 1, has level of distribution $\vartheta = 1$, and is congruence uniform.}
	\end{proposition}
	
	\begin{proof}
		{The claim that the sequence has index 1 follows from classical results of Gelfond \cite{Ge} which implies $N(x) \sim x/2$ -- see Theorem A in \cite{Sp}. The claim that the sequence has level of distribution $\vartheta = 1$ follows from Theorem 1.1 of \cite{Sp}. And again by Remark \ref{remark: ABC}, congruence uniformity is trivial.}
	\end{proof}

	\subsection{Main results}
	\label{subsec:mainresults}
	
	Our main results depend on the following setup. As above we let $(a_n)$ be a sequence of nonnegative numbers, and for a parameter $x$ we let $u$ be a random integer such that 
	$$
	\PP(u = m) = \frac{a_m}{N(x)} \mathbf{1}[1 \leq m \leq x].
	$$
	In the case that $a_n$ is the indicator function of a subset of natural numbers, $u$ is uniformly distributed on elements of the set no more than $x$. As before, we let $p_1 \geq p_2 \geq \cdots $ be the prime factors of $u$ listed with multiplicity, with the convention that $p_j = 1$ if $n$ has fewer than $j$ prime factors.
	
	We will prove results comparing the distribution of $p_1,p_2,...$ to a Poisson-Dirichlet process (Theorem \ref{thm:poisson_dirichlet} and Lemma \ref{lem:corr_func}) and also an upper bound for the likelihood that $p_1 = P^+(u)$ is exceptionally large (Theorem \ref{thm:upperbound}). These results are related in that they use almost the same information, but because they may be of independent interest we have written this note so that their proofs may be read independently.
	
	\begin{theorem}
		\label{thm:poisson_dirichlet}
		If $(a_n)$ is {$1$ well-distributed}, then
		$$
		\mathbb{P}\Big(\frac{\log p_1}{\log u} \leq c_1,..., \frac{\log p_k}{\log u} \leq c_k\Big) \rightarrow \mathbb{P}(L_1 \leq c_1,...,L_k \leq c_k),
		$$
		as $x\rightarrow\infty$. That is, the process $\tfrac{\log p_1}{\log u}, \tfrac{\log p_2}{\log u},...$ tends in distribution to the Poisson-Dirichlet process $L_1, L_2,...$.
	\end{theorem}
	
	This generalizes to multiple prime factors a result noted by Granville \cite{Gr} (with details of the proof provided by Wang \cite{Wa}) that for shifted primes the Elliott-Halberstam conjecture implies the distribution of the largest prime divisor is governed by the Dickman function (a phenomenon first conjectured by Pomerance \cite{Po}).  And this result proves unconditionally that large prime factors of the Thue-Morse tend to a Poisson-Dirichlet distribution.
	
	On the other hand, while the values of irreducible polynomials do not appear to have level of distribution $1$, it is reasonable to believe that their prime factors tend to a Poisson-Dirichlet distribution. That the distribution of the largest prime factor is governed by the Dickman function was given a conditional proof by Martin \cite{Ma} on the assumption of a prime number theorem for polynomial sequences. It may be possible to formulate a relaxed version of level of distribution $1$ which applies to the values of irreducible polynomials and which also implies a Poisson-Dirichlet distribution for large prime factors, but we do not pursue this further here. (Indeed the condition $(D_\sigma)$ in the very recent preprint \cite{Mo} might be a correct starting point.)

	Even for a sequence with level of distribution less than 1, one may still compare the \emph{correlation functions} of its large prime factors to those of the Poisson-Dirichlet process, at least against test functions with restricted support.
	
	\begin{lemma}
		\label{lem:corr_func}
		If $(a_n)$ {is $\sigma$ well-distributed}, then for any ${k \geq 1}$ and any continuous 
        $\eta: \R^k \rightarrow \mathbb{C}$ with $\supp \eta \subset \{y \in \R_+^k:\, y_1 + \cdots + y_k < \sigma\}$ {(so in particular $\eta$ is compactly supported)},
		\begin{equation}
			\label{eq:partial_corr}
			\EE \sum_{\substack{j_1,...,j_k \\ \textrm{distinct}}} \eta\Big(\frac{\log p_{j_1}}{\log u},..., \frac{\log p_{j_k}}{\log u}\Big) \rightarrow \EE \sum_{\substack{j_1,...,j_k \\ \textrm{distinct}}} \eta(L_{j_1},...,L_{j_k}),
		\end{equation}
		as $x\rightarrow\infty$.
	\end{lemma}
	
    Here and throughout the paper we adopt the convention that $\R_+ = (0,\infty)$. So $\eta$ being supported in $\R_+^k$ means that $\eta(y_1,...,y_k)$ will vanish when any $y_i$ is sufficiently close to $0$. (Recall that the support of a function is the \emph{closure} of the set on which it does not vanish.)

    \begin{remark}
    In Theorem \ref{thm:poisson_dirichlet} and Lemma \ref{lem:corr_func} and in other places in this paper it should be possible to adopt a weaker definition of level of distribution, in which the bound \eqref{eq:distribution-eqn} need only hold for sums over $d$ in which $\Omega(d) \leq k$ and all prime factors of $d$ are larger than $x^\epsilon$, with implicit constants depending on $k$ and $\epsilon$, for all $k \geq 1$ and $\epsilon > 0$, but we do not pursue this generalization here.
    \end{remark}
	
	In fact, the right hand side in \eqref{eq:partial_corr} has a simple evaluation in general: for any continuous $\eta$ {with compact support in $\mathbb{R}_+^k$},
	\begin{equation}
		\label{eq:PD_correlations}
		\EE \sum_{\substack{j_1,...,j_k \\ \textrm{distinct}}} \eta(L_{j_1},...,L_{j_k}) = \int_{\R_+^k} \frac{\mathbf{1}[t_1+\cdots + t_k \leq 1]}{t_1\cdots t_k} \eta(t)\, d^kt.
	\end{equation}
	This is \cite[(4)]{Ta} (see also the closely related \cite[(14)]{ArKoMi}).
	
	Theorem \ref{thm:poisson_dirichlet} will be seen to follow from Lemma \ref{lem:corr_func} and a characterization of the Poisson-Dirichlet process due to Arratia-Kochman-Miller \cite{ArKoMi}.

    \begin{remark}
    It may be worthwhile for number theorists unfamiliar with correlation sums to write out an explicit example of the sort of sum which appears on the {left} hand side of \eqref{eq:partial_corr}. For instance if $u = p_1 p_2 p_3$ with $p_1 \geq p_2 \geq p_3$ all primes then we have for the $2$-point correlation sum,
    \begin{align*}
    \sum_{\substack{j_1,j_2 \\ \textrm{distinct}}} \eta\Big(\frac{\log p_{j_1}}{\log u}, \frac{\log p_{j_2}}{\log u}\Big) = &\eta\Big(\frac{\log p_1}{\log u}, \frac{\log p_2}{\log u}\Big) + \eta\Big(\frac{\log p_1}{\log u}, \frac{\log p_3}{\log u}\Big) \\
    &+\eta\Big(\frac{\log p_2}{\log u}, \frac{\log p_1}{\log u}\Big)+ \eta\Big(\frac{\log p_2}{\log u}, \frac{\log p_3}{\log u}\Big)\\
    &+ \eta\Big(\frac{\log p_3}{\log u}, \frac{\log p_1}{\log u}\Big) + \eta\Big(\frac{\log p_3}{\log u}, \frac{\log p_2}{\log u}\Big).
    \end{align*}
    We have adopted the convention for such a $u$ that $p_j=1$ for $j\geq 4$, but because of the support of $\eta$ such terms do not appear in this sum. Note that the sum is symmetric in $p_1, p_2$ and $p_3$.
    
    The left hand side of \eqref{eq:partial_corr} will then be an average over $u$ of sums of this type. 
	
    Lemma \ref{lem:corr_func} has a surface-level resemblance to results that can be proven about zeros of $L$-functions. The reader unfamiliar with correlation sums as occur in the Lemma may consult \cite[Ch.1]{HoKrPeVi} for a general introduction and further information.      
    \end{remark}
	
	Lemma \ref{lem:corr_func} gives information about prime divisors of intermediate size, but because of restrictions on the support of $\eta$ it does not entail an asymptotic formula for the distribution of the largest prime factor of $u$. Our last result shows that even this partial information about the level of distribution entails an upper bound for how often the largest prime factor can be especially large.
	
	\begin{theorem}
		\label{thm:upperbound}
		If $(a_n)$ is {$\sigma$ well-distributed for some $\sigma > 0$}, then for any $\epsilon > 0$,
		$$
		\limsup_{x\rightarrow\infty}\PP(P^+(u) \geq u^{1-\epsilon}) \ll \epsilon,
		$$
		where the implicit constant depends on the sequence $(a_n)$.
	\end{theorem} 
	
	We note that this is an essentially optimal result, since for sequences with level of distribution $1$ by Theorem \ref{thm:poisson_dirichlet} we have $\PP(P^+(u) \geq u^{1-\epsilon}) \sim 1- \rho(1/(1-\epsilon))$, and for $\epsilon$ small enough that $1\leq 1/(1-\epsilon) \leq 2$ we have $1- \rho(1/(1-\epsilon)) = \log(1/(1-\epsilon)) \approx \epsilon$ (see \cite[(7.10)]{MoVa} for the evaluation of $\rho$ in this range).
	
	In the case of sampling shifted primes $p-1 \leq x$, Theorem \ref{thm:upperbound} recovers the following corollary,
 
 	\begin{corollary}[Erd\H{o}s]
		\label{cor:shiftedprimes}
		For any $\epsilon > 0$,
		$$
		\limsup_{x\rightarrow\infty} \frac{1}{\pi(x)}|\{p \leq x: P^+(p-1) \geq p^{1-\epsilon}\}| \ll \epsilon.
		$$
	\end{corollary}

    This result appears implicitly, though somewhat obscurely, in a paper of Erd\H{o}s (see the line beginning with ``the sum in $a$ is less than" in \cite[p. 213]{Er}). A recent paper of Ding \cite{Di2} gives a proof with explicit constants, and explains some of the history around this estimate.

	As in these other proofs of Corollary \ref{cor:shiftedprimes}, the proof of Theorem \ref{thm:upperbound} relies on an upper bound sieve.
	
    \subsection{Acknowledgements}

    We thank Yuchen Ding, Ofir Gorodetsky, Ram Murty, and anonymous referees for comments, suggestions and corrections to earlier versions of this manuscript. We have used ChatGPT 5.2 to proofread a version of the manuscript and to help format references but no parts of the paper were computer generated. B.R. is supported by an NSERC grant. A.B. is supported by a Coleman Postdoctoral Fellowship.
 
	\section{Resemblance to Poisson-Dirichlet: a proof of Theorem \ref{thm:poisson_dirichlet} and Lemma \ref{lem:corr_func}}

	\begin{proof}[Proof of Lemma \ref{lem:corr_func}]
        Let us rewrite the left hand side of \eqref{eq:partial_corr} as
        \begin{equation}
        \label{eq:partial_corr_sum}
        \frac{1}{N(x)} \sum_{n\leq x} a_n \sum_{\substack{j_1,...,j_k \\ \textrm{distinct}}} \eta\Big(\frac{\log p_{j_1}}{\log n},..., \frac{\log p_{j_k}}{\log n}\Big),
        \end{equation}
        where in the inner sum $p_1, p_2,...$ are the prime factors of $n$ listed according to multiplicity. We have given an expression for the right hand side of \eqref{eq:partial_corr} in the evaluation \eqref{eq:PD_correlations}. We will show convergence to this expression in the following steps: first we show that $\log n$ can be replaced by $\log x$ in the denominators above, second we show that those $n$ with repeated large prime factors contribute only a negligible amount to the sum, and third we show that the resulting expression can be expressed in more conventional language of analytic number theory and in that way easily evaluated.

        Let us note from the start there are constants {$a > 0$} and {$c < \sigma$} such that $\eta(y_1,...,y_k)$ vanishes whenever {$y_i \leq a$} or $y_1+\cdots+y_k \geq c$; this is because of the restricted support of $\eta$. 

        Thus in our first step we show that \eqref{eq:partial_corr_sum} is
        \begin{equation}
        \label{eq:corr_n_to_x}
        =\frac{1}{N(x)} \sum_{n\leq x} a_n \sum_{\substack{j_1,...,j_k \\ \textrm{distinct}}} \eta\Big(\frac{\log p_{j_1}}{\log x},..., \frac{\log p_{j_k}}{\log x}\Big) + o_{x\rightarrow\infty}(1).
        \end{equation}
        To see this, note that in \eqref{eq:partial_corr_sum} for each tuple $j_1,...,j_k$ for which the summand is non-vanishing, $p_{j_1},...,p_{j_k} \geq n^a$. But $n$ can have no more than $\lfloor 1/a \rfloor$ prime factors $p_j \geq n^a$. Thus we have the following crude bound: for any $n$,
        \begin{equation}
        \label{eq:crude_bound_corr}
        \Big| \sum_{\substack{j_1,...,j_k \\ \textrm{distinct}}} \eta\Big(\frac{\log p_{j_1}}{\log n},..., \frac{\log p_{j_k}}{\log n}\Big) \Big| \leq \max(|\eta|) \cdot k! \binom{\lfloor 1/a \rfloor}{k} = O(1),
        \end{equation}
        where the implicit constant depends on $k$ and $\eta$ but does not depend on $n$.

        Hence for arbitrary $\varepsilon > 0$, the left hand side of \eqref{eq:partial_corr_sum} is
        \begin{equation}
	\label{eq:truncated_sum}
        = \frac{1}{N(x)}\sum_{x^{1-\epsilon} < n \leq x} a_n       \sum_{\substack{j_1,...,j_k \\ \textrm{distinct}}} \eta\Big(\frac{\log  p_{j_1}}{\log n},..., \frac{\log p_{j_k}}{\log n}\Big)  + O\Big(\frac{N(x^{1-\epsilon})}{N(x)}\Big).
	\end{equation}
        {Due to continuity and compact support, $\eta$ is uniformly continuous. Hence for any $p_{j_1},...,p_{j_k}$ in the sum above, for $x^{1-\epsilon} < n \leq x$,
	\begin{align*}
        \eta\Big(\frac{\log p_{j_1}}{\log n},..., \frac{\log p_{j_k}}{\log n}\Big) &= \eta\Big(\frac{\log p_{j_1}}{\log x} +O(\epsilon),..., \frac{\log p_{j_k}}{\log x} +O(\epsilon)\Big) \\
        &=\eta\Big(\frac{\log p_{j_1}}{\log x},..., \frac{\log p_{j_k}}{\log x}\Big) + O(\delta(\epsilon)),
	\end{align*}
    for a quantity $\delta(\epsilon) \rightarrow 0$ as $\epsilon \rightarrow 0$.
        Thus \eqref{eq:truncated_sum} is
        \begin{multline}
        \label{eq:trunc_sum_with_x}
        = \frac{1}{N(x)}\sum_{x^{1-\epsilon} < n \leq x} a_n \sum_{\substack{j_1,...,j_k \\ \textrm{distinct}}} \eta\Big(\frac{\log p_{j_1}}{\log x},..., \frac{\log p_{j_k}}{\log x}\Big) \\
        + O\Big(\delta(\epsilon)\cdot\frac{N(x)-N(x^{1-\epsilon})}{N(x)}\Big) +  O\Big(\frac{N(x^{1-\epsilon})}{N(x)}\Big).
	\end{multline}
        This is because in both \eqref{eq:truncated_sum} and \eqref{eq:trunc_sum_with_x}, for each $n$ only a bounded number of tuples $j_1,...,j_k$ will give rise to a nonzero summand (there will be at most $k! \binom{\lfloor 1/a \rfloor}{k}$ such tuples, as in \eqref{eq:crude_bound_corr}), and in the the difference between such summands in \eqref{eq:truncated_sum} and \eqref{eq:trunc_sum_with_x} will be $O(\delta(\epsilon))$ always.

        An index $\alpha \geq \sigma > 0$ implies the error terms in \eqref{eq:trunc_sum_with_x} are $O(\delta(\epsilon))+ o_{x\rightarrow\infty}(1)$, and because $\delta(\epsilon)$ can be taken arbitrarily small, this implies \eqref{eq:corr_n_to_x}.}

        In our second step we show that in \eqref{eq:corr_n_to_x}, the sum over $n$ can be replaced by a sum only over those $n$ which have no repeated large prime factors. Let us consider the complementary set of $n$ which do have a repeated large prime factor; define $S(x)$ to be the set of positive integers $n \leq x$ such that in the prime factorization $n = p_1 p_2 \cdots$ some $p_i \in [x^a, x^c]$ occurs with multiplicity at least $2$. We have
        { 
        \begin{multline}
        \label{eq:repeated_bound}
        \sum_{n\in S(x)} a_n \leq \sum_{n\leq x} a_n \sum_{x^a \leq p \leq x^c} \mathbf{1}[\,p^2|n\,] = \sum_{x^a \leq p \leq x^c} N_{p^2}(x) \\
       \ll N(x) \sum_{x^a \leq p \leq x^c} \frac{1}{p^2} (\log x)^B + \log(x)^B \sum_{x^a \leq p \leq x^c} 1 \\
        = N(x) \log(x)^B \sum_{x^a \leq p \leq x^c} \frac{1}{p^2} + O_B\Big(x^c(\log x)^B\Big) = o(N(x)),
        \end{multline}
        where the second to last equation holds for some $B>0$ and follows from the congruence uniformity property \ref{PROP=Locally-Uniform}, and the last line follows from the prime number theorem and the assumption \ref{PROP=index} that the sequence has an index satisfying $\alpha \geq \sigma > c$.}
        
        Furthermore, observe that the crude bound \eqref{eq:crude_bound_corr} remains true if $\log n$ is replaced by $\log x$ in the denominators, for all $n\leq x$; the argument remains the same as in \eqref{eq:crude_bound_corr}. Hence using this estimate and \eqref{eq:repeated_bound} we see \eqref{eq:corr_n_to_x} is
        \begin{equation}
        \label{eq:corr_n_to_x_norepeated}
        =\frac{1}{N(x)} \sum_{\substack{n\leq x \\ n \notin S(x)}} a_n \sum_{\substack{j_1,...,j_k \\ \textrm{distinct}}} \eta\Big(\frac{\log p_{j_1}}{\log x},..., \frac{\log p_{j_k}}{\log x}\Big) + o_{x\rightarrow\infty}(1).
        \end{equation}

        Now coming to the third and final step, we observe that \eqref{eq:corr_n_to_x_norepeated} can be rewritten, 
        $$
        =\frac{1}{N(x)} \sum_{\substack{n\leq x \\ n \notin S(x)}} a_n \sum_{\substack{p_{j_1},...,p_{j_k} \\ \textrm{distinct}}} \eta\Big(\frac{\log p_{j_1}}{\log x},..., \frac{\log p_{j_k}}{\log x}\Big) + o_{x\rightarrow\infty}(1).
        $$
        The distinction between this sum and \eqref{eq:corr_n_to_x_norepeated} is that now the prime factors $p_{j_1},...,p_{j_k}$ must be distinct, whereas before we needed only that the indices $j_1,...,j_k$ be distinct. But because $n\notin S(x)$ in the sum and because of the support of $\eta$, these coincide whenever the inner summand is non-vanishing.

        But there is a one-to-one correspondence between tuples $(p_{j_1},...,p_{j_k})$ of prime factors of $n$, all distinct, and tuples $(q_1,...,q_k)$ of distinct primes in which $q_1\cdots q_k | n$. So we can further rewrite the above as
        \begin{equation}
        \label{eq:q_corr}
        =\frac{1}{N(x)} \sum_{\substack{n\leq x \\ n \notin S(x)}} a_n \sum_{\substack{q_1,...,q_k \\ \textrm{prime, distinct}}} \mathbf{1}\big[\,q_1\cdots q_k | n\,\big] \eta\Big(\frac{\log q_1}{\log x},\cdots ,\frac{\log q_k}{\log {x}}\Big) +o_{x\rightarrow\infty}(1).
        \end{equation}
        By the same argument as in \eqref{eq:crude_bound_corr},
        $$
        \sum_{\substack{q_1,...,q_k \\ \textrm{prime, distinct}}} \mathbf{1}\big[\,q_1\cdots q_k | n\,\big] \eta\Big(\frac{\log q_1}{\log x},\cdots ,\frac{\log q_k}{\log x}\Big) = O(1)
        $$
        uniformly in $n$. So we may use the bound \eqref{eq:repeated_bound} for contributions from $n\in S(x)$ to see that \eqref{eq:q_corr} is
        $$
        =\frac{1}{N(x)} \sum_{n\leq x } a_n \sum_{\substack{q_1,...,q_k \\ \textrm{prime, distinct}}} \mathbf{1}\big[\,q_1\cdots q_k | n\,\big] \eta\Big(\frac{\log q_1}{\log x},\cdots ,\frac{\log q_k}{\log x}\Big) +o_{x\rightarrow\infty}(1).
        $$
        But this is
        $$
        = \frac{1}{N(x)} \sum_{\substack{q_1,...,q_k \\ \textrm{prime, distinct}}} N_{q_1\cdots q_k}(x) \eta\Big(\frac{\log q_1}{\log x},\cdots ,\frac{\log q_k}{\log x}\Big) + o_{x\rightarrow\infty}(1).
	$$
        {We now simplify the above sum using} {that the level of distribution is $\vartheta \geq \sigma > c$}. Note that {the above} sum can be {rewritten} as 
        {$$
        = \frac{k!}{N(x)} \sum_{\substack{q_1<q_2<\dots<q_k \\ \textrm{prime, distinct}}} N_{q_1...q_k}(x) \eta\Big(\frac{\log q_1}{\log x},\cdots ,\frac{\log q_k}{\log x}\Big) + o_{x\rightarrow\infty}(1)
        $$}
        {Since $q_1...q_k \le x^c$, {one may think of the above sum as occurring over a subset of integers less than or equal to $x^c$}, and using \eqref{eq:distribution-eqn}, this simplifies to} \\
                
        \begin{equation}
        \label{eq:sum_with_g}
        =\sum_{\substack{q_1,...,q_k \\ \textrm{prime, distinct} }} g(q_1\cdots q_k) \eta\Big(\frac{\log q_1}{\log x},\cdots ,\frac{\log q_k}{\log x}\Big) + o_{x\rightarrow\infty}(1).
	    \end{equation}

        Observe that the upper bound in \eqref{eq:g_asymp} implies that for any $j\geq 2,$
        $$
        \sum_{\substack{q\geq x^a \\ \textrm{prime}}} g(q)^j = o_{x\rightarrow\infty}(1).
        $$
        Utilizing the multiplicativity of $g$ and then this bound, we see that \eqref{eq:sum_with_g} is
        \begin{multline}
	\label{eq:g_separated}
        =\sum_{\substack{q_1,...,q_k \\ \textrm{prime, distinct}}} g(q_1)\cdots g(q_k) \eta\Big(\frac{\log q_1}{\log x},\cdots,\frac{\log q_k}{\log x}\Big) + o_{x\rightarrow\infty}(1)  \\ =\sum_{\substack{q_1,...,q_k \\ \textrm{prime}}} g(q_1)\cdots g(q_k) \eta\Big(\frac{\log q_1}{\log x},\cdots, \frac{\log q_k}{\log x}\Big) + o_{x\rightarrow\infty}(1).
	\end{multline}

		But finally the asymptotic formula in \eqref{eq:g_asymp} and partial summation implies that if $\nu$ is the indicator function of an interval {$[A,B] \subset (0,\infty)$},
		$$
		\sum_{q\; \textrm{prime}} g(q)\, \nu\Big(\frac{\log q}{\log x}\Big) = \int_{\R_+} \frac{\nu(t)}{t}\, dt + {o_{x \to \infty}(1).}
		$$
		This implies that if $\eta$ is the indicator function of a rectangle $[A_1,B_1]\times\cdots\times[A_k,B_k] \subset (0,\infty)^k$, the {right} hand side of \eqref{eq:g_separated} is
		\begin{multline}
			\label{eq:final_asymp}
			=\left(\sum_{q_1 \textrm{ prime}} g(q_1)\mathbf{1}_{[A_1,B_1]}\Big(\frac{\log q_1}{\log x}\Big)\right)\cdots \left(\sum_{q_k\textrm{ prime}} g(q_k) \mathbf{1}_{[A_k,B_k]}\Big(\frac{\log q_k}{\log x}\Big)\right) + o_{x \to \infty}(1)\\
            = \int_{\R_+^k} \frac{1}{t_1\cdots t_k} \eta(t)\, d^kt + o_{x \to \infty}(1).
		\end{multline}
		But because linear combinations of such functions are dense in the space of continuous functions with compact support in $\R_+^k$, a standard approximation argument implies \eqref{eq:final_asymp} is true for this class of functions as well.
		
        This gives an asymptotic evaluation of the left hand side of \eqref{eq:partial_corr_sum}. Due to the restricted support of $\eta$, the indicator function in the expression \eqref{eq:PD_correlations} for Poisson-Dirichlet plays no role here, and we have therefore that the left hand side tends to the right hand side of \eqref{eq:partial_corr_sum}.
	\end{proof}
	
	We will verify Theorem \ref{thm:poisson_dirichlet} using the above lemma and the following criterion, from \cite[Lemma 2]{ArKoMi}:

	\begin{lemma}[Arratia-Kochman-Miller]
	\label{lem:akm}
		If for each $x$, $(L_1(x),L_2(x),...)$ is a random process with $L_1(x) \geq L_2(x) \geq \cdots \geq 0$ satisfying $\sum L_i(x) = 1$, and if for any collection of disjoint intervals $I_i = [a_i,b_i] \subset (0,1]$ with $b_1+\cdots + b_k < 1$ we have
		\begin{equation}
			\label{eq:AKM}
			\liminf_{x\rightarrow\infty} \EE \prod_{i=1}^k |\{j:\, L_j(x) \in I_i\}| \geq \prod_{i=1}^k \log(b_i/a_i) = \int _{I_1\times \cdots \times I_k} \frac{d^k t}{t_1\cdots t_k},
		\end{equation}
		then the process $L_1(x), L_2(x),...$ tends in distribution to the Poisson-Dirichlet process $L_1, L_2, ...$ as $x\rightarrow\infty$.
	\end{lemma}
	
	\begin{proof}[Proof of Theorem \ref{thm:poisson_dirichlet}]
		Let $L_i(x) = \log p_i/\log u$. If $\eta^\ast = \mathbf{1}_{I_1\times \cdots \times I_k}$ were continuous this would be implied by Lemma \ref{lem:corr_func}, as
		$$\prod_{i=1}^k |\{j:\, L_j(x) \in I_i\}| = \sum_{\substack{j_1,...,j_k \\ \textrm{distinct}}} \eta^\ast\Big(\frac{\log p_{j_1}}{\log u},..., \frac{\log p_{j_k}}{\log u}\Big).$$
		But for any $\epsilon > 0$, we can find a continuous function $\eta^-$ with support in $\{y\in \R_+^k:\, y_1 + \cdots + y_k < 1\}$ such that 
		$$\eta^\ast \geq \eta^-$$ 
		and
		$$\int_{\R_+^k} (\eta^\ast - \eta^-) \, \frac{d^k t}{t_1\cdots t_k} \leq \epsilon.$$
		{(Indeed, if a box $I_1^-\times  \cdots \times I_k^-$ is inside $I_1\times \cdots \times I_k$ and only slightly smaller, a continuous function $\eta^-$ sandwiched between indicator functions will satisfy these conditions.)} Thus lower-bounding $\eta^\ast$ by $\eta^-$, {applying Lemma \ref{lem:corr_func} for $\eta^-$, and then using the correlation function formula \eqref{eq:PD_correlations}}, we have
		$$\liminf_{x\rightarrow\infty} \EE \sum_{\substack{j_1,...,j_k \\ \textrm{distinct}}} \eta^\ast\Big(\frac{\log p_{j_1}}{\log u},..., \frac{\log p_{j_k}}{\log u}\Big) \geq \int_{\R_+^k} \eta^- \frac{d^k t}{t_1\cdots t_k}.$$
		
		But the right hand side is within $\epsilon$ of 
		$$\int \eta^\ast \frac{d^k t}{t_1\cdots t_k} = \int_{I_1\times \cdots I_k} \frac{d^kt}{t_1\cdots t_k},$$
		and because $\epsilon$ is arbitrary this verifies \eqref{eq:AKM} is true in this case, and the result follows. 
	\end{proof}

	\begin{remark}
	In effect, what Lemma \ref{lem:akm} of Arratia-Kochman-Miller shows is that if the result of Lemma \ref{lem:corr_func} holds for $\sigma = 1$ for some ordered process, then that process tends in distribution to the Poisson-Dirichlet process. That is, the convergence of correlation functions implies convergence in distribution in this context.
	\end{remark}

	\section{Upper bounds on largest primes: a proof of Theorem \ref{thm:upperbound}}
	
	In proving Theorem \ref{thm:upperbound} we will use the Selberg upper bound sieve. We recall the setup from \cite{FrIw}. 
	
	Let $\mathcal{P}$ be a finite set of primes, and define $P = \prod_{p \in \mathcal{P}}p$. Let $g(d) \in [0,1)$ be defined for $d|P$ and be a multiplicative function for this set of $d$. In the notation \eqref{eq:N_def} as before, define 
	$$
	r_d = N_d(x) - g(d)N(x),
	$$ 
	and suppose there are constants $\kappa, K > 0$ such that
	\begin{equation}
		\label{eq:upperbound_assumption}
		\prod_{\substack{w \leq p < z \\ p \in \mathcal{P}}} \frac{1}{1-g(p)} \leq K \Big(\frac{\log z}{\log w}\Big)^\kappa,
	\end{equation}
	for all $z > w \geq 2$.
	
	\begin{theorem}[An explicit upper bound sieve]
		\label{thm:sieve}
		For $\mathcal{P}$, $P$, $g(d)$, and $r_d$ as just described, with $g$ satisfying \eqref{eq:upperbound_assumption}, define $k = \kappa + \log K$. If a parameter $D$ is chosen such that all $p \in \mathcal{P}$ satisfy $p < D^{1/(4k)}$, then
		$$
		\sum_{\substack{n\leq x \\ (n,P)=1}} a_n \leq C\cdot V \cdot N(x) + \sum_{\substack{d|P \\ d < D}} \tau_3(d) |r_d|,
		$$
		where $C > 0$ is a constant which depends only on $k$, we have defined 
		$$
		V = \prod_{p \in \mathcal{P}} (1-g(p)),
		$$
		and $\tau_3(d) = \sum_{d_1 d_2 d_3 = d} 1$ is the threefold divisor function.
	\end{theorem}
	
	\begin{proof}
		This is Theorem 7.4 in \cite{FrIw}, where in their notation we have taken $s = 4k$ and $X = N(x)$. (Note that the hypothesis of their theorem requires $g(d) \in (0,1)$ for $d|P$, but one may check the proof works with no modification if $g(d) \in [0,1)$ for $d|P$.)
	\end{proof}
	
	We now apply this result to get an upper bound on the frequency with which a number $n$ has a prime factor larger than $n^{1-\epsilon}$. {It is only when $\epsilon$ is small that Theorem \ref{thm:upperbound} is nontrivial so we may suppose with no loss of generality that $\epsilon < 1/2$.} The idea behind the proof is easy to state: if $n$ has a prime factor larger than $n^{1-\epsilon}$, it will have no prime factors in between $n^\epsilon$ and $n^{1-\epsilon}$. {(Since $\epsilon < 1/2$, we have $n^\epsilon < n^{1-\epsilon}$.)} The upper bound is obtained by sieving by (a subset of) such primes. 
	
	\begin{proof}[Proof of Theorem \ref{thm:upperbound}]
		Note that by \eqref{eq:g_asymp} there is a constant $z_0$ such that $g(p) \leq 1/2$ for all $p \geq z_0$. Further by \eqref{eq:g_asymp} we have,
		$$
		\prod_{z_0 \leq p < z}(1-g(p)) = \exp[-\log \log z + O(1)],
		$$
		so that under the hypothesis of Theorem \ref{thm:upperbound}, we have that \eqref{eq:upperbound_assumption} is satisfied for any subset $\mathcal{P}$ of primes larger than $z_0$, for $\kappa = 1$ and some constant $K$. As in Theorem \ref{thm:sieve} define $k = \kappa + \log K$.
		
		Let $\alpha > 0$ be the index. Using {$N(x^{1/2}) = x^{\alpha/2 + o(1)} = o(N(x))$, we have}
		\begin{equation}
			\label{eq:firstbound}
			\sum_{n\leq x} a_n \mathbf{1}[P^+(n) \geq n^{1-\epsilon}] = \sum_{x^{1/2} < n\leq x} a_n \mathbf{1}[P^+(n) \geq n^{1-\epsilon}] + {o_{x\rightarrow\infty}(N(x))}.
		\end{equation}
		But if $P^+(n) \geq n^{1-\epsilon}$ then $n$ is not divisible by any primes strictly in between $n^\epsilon$ and $n^{1-\epsilon}$. If $x^{1/2} < n \leq x$, then this {implies} such $n$ is not divisible by any primes strictly in between $x^\epsilon$ and $x^{(1-\epsilon)/2}$. 
		
		We will sieve out by a sparser set of primes even than these. Let $\delta > 0$ be some number smaller than $\sigma$. (So $(a_n)$ has level of distribution and index larger than $\delta$.)
		
		Let $D = x^\delta$ and then set $\delta_0 = \min[(1-\epsilon)/2, \delta/(4k)]$. Now define $\mathcal{P}$ to be the primes larger than $z_0$ and strictly in between $x^\epsilon$ and $x^{\delta_0}$. (Let $\mathcal{P}$ be empty if there are no such primes.) We have $\mathcal{P}$ is a subset of the primes in between $x^\epsilon$ and $x^{(1-\epsilon)/2}$, and also all $p\in \mathcal{P}$ satisfy $p < D^{1/4k}$. 
		
		Thus, if we set $P = \prod_{p\in\mathcal{P}}p$, the right hand side of \eqref{eq:firstbound} is
		\begin{align*}
			&\leq \sum_{\substack{n\leq x \\ (n,P)=1}} a_n + o_{x\rightarrow\infty}(N(x)) \\
			&\leq C' \cdot V \cdot N(x) + \sum_{\substack{d < D \\ d|P}} \tau_3(d)|r_d| + o_{x\rightarrow\infty}(N(x)),
		\end{align*}
		where $C'$ is a constant which depends only on the sequence $(a_n)$.
		
		Now note that for sufficiently small $\epsilon$, once $x$ is sufficiently large, the set $\mathcal{P}$ will not be empty. With no loss of generality we may assume $\epsilon$ is this small and $x$ is at least this large in the remainder of the proof.
		
		We have
		$$
		V = \prod_{x^\epsilon < p < x^{\delta_0}} (1-g(p)) = \exp[-\log \log(x^{\delta_0}) + \log \log(x^\epsilon) + O(1)] \ll \epsilon,
		$$
		Moreover, there are constants $B$ and $C$ such that
		\begin{equation}
		\label{eq:r_bound}
		|r_d| \ll \frac{(\log x)^B C^{\Omega(d)}}{d} N(x) + (\log x)^B C^{\Omega(d)} \ll \frac{(\log x)^B C^{\Omega(d)}}{d} N(x)
		\end{equation}
		for $d \leq D$. {(The first inequality follows from congruence uniformity, and the second from the index relation $N(x)/d \gg x^{\alpha+o(1)}/x^\delta \gg 1$.)} Taking such a constant $C$, we note
		\begin{equation*}
			\sum_{\substack{d < D \\ d|P}} \tau_3(d) |r_d| \leq \Big( \sum_{\substack{d\leq x^\delta \\ d|P}} \frac{C^{\Omega(d)}\tau_3(d)^2}{d}\Big)^{1/2} \Big( \sum_{\substack{d\leq x^\delta \\ d|P}} \frac{d}{C^{\Omega(d)}} |r_d|^2\Big)^{1/2}.
		\end{equation*}
		Note $\tau_3(n) \leq 3^{\Omega(n)}$ and for sufficiently large $x$ we have $(9C)^{\Omega(d)}/d \leq 1$ for all $d | P$. So using Lemma \ref{lem:estimate-C} to estimate the first parentheses and \eqref{eq:r_bound} to estimate the second, for sufficiently large $x$ the above is
		\begin{equation*}
			\ll (\log x)^A \Big( (\log x)^B N(x) \sum_{d\leq x^\delta} |r_d|\Big)^{1/2}
		\end{equation*}
		for some constant $A > 0$.

		Using that $(a_n)$ has level of distribution greater than $\delta$, the above is
		\begin{equation*}
		\ll N(x)/(\log x)^{A'} = o_{x\rightarrow\infty}(N(x)),
		\end{equation*}
		for any constant $A' > 0$.
		
		Putting matters together we have
		$$
		\sum_{n\leq x} a_n \mathbf{1}[P^+(n) \geq n^{1-\epsilon}] \ll \epsilon N(x) + o_{x\rightarrow\infty}(N(x)),
		$$
		where the implicit constant depends only on the sequence $(a_n)$, which implies the Theorem.
	\end{proof}

	\begin{remark}
	    Theorem \ref{thm:upperbound} says that the likelihood that $\log P^+(u)/\log u\geq 1-\epsilon$ is $O(\epsilon)$. Although in its proof we have imported Theorem \ref{thm:sieve} directly from sieve theory, it is likely possible and would be interesting to abstract the combinatorial content of this sieve bound to prove a version of Theorem \ref{thm:upperbound} for general point processes on the simplex with correlation functions known to agree with those of a Poisson-Dirichlet process against test functions with restricted support, in the sense of Lemma \ref{lem:corr_func}. We do not pursue this here however.
	\end{remark}

	\Addresses
	
\end{document}